\documentclass{amsart}
\usepackage[utf8]{inputenc}
\usepackage[english]{babel}

\usepackage[margin=3.0cm]{geometry}
\usepackage{enumerate}
\usepackage{float}
\usepackage[final]{graphicx}
\usepackage{epstopdf} 

\usepackage{tabularx}
\newcolumntype{C}[1]{>{\centering\arraybackslash}m{#1}}

\usepackage{caption}
\usepackage{subcaption}
\usepackage{enumitem}
\usepackage{comment}
\usepackage{soul}

\usepackage{tikz-cd}
\usepackage{amsmath, multirow}
\usepackage{amsfonts}
\usepackage{amssymb}
\usepackage{amsthm}
\usepackage{hyperref}
\usepackage{cleveref}

\DeclareCaptionSubType[alph]{figure}
\usepackage[T1]{fontenc}
\usepackage{mathrsfs}
\usepackage{mathtools}

\usepackage{tikz}
\usetikzlibrary{decorations.markings}
\usetikzlibrary{decorations.pathreplacing}
\usetikzlibrary{arrows,shapes,positioning,patterns}
\usetikzlibrary{knots}
\tikzstyle{none}=[inner sep=0pt]
\pgfdeclarelayer{edgelayer}
\pgfdeclarelayer{nodelayer}
\pgfsetlayers{edgelayer,nodelayer,main}
\usepackage{url} 
\usepackage[all]{xy}

\newcommand{\R}{\mathcal R}

\newcommand{\bZ}{\mathbb Z}

\renewcommand{\H}{\mathrm{H}}

\newcommand{\x}{\times}

\theoremstyle{plain}
\newtheorem*{thmintro}{Theorem}
\newtheorem{thm}{Theorem}[section]
\newtheorem{prop}[thm]{Proposition}
\newtheorem{lem}[thm]{Lemma}
\newtheorem{cor}[thm]{Corollary}
\theoremstyle{remark}
\newtheorem{rem}[thm]{Remark}

\theoremstyle{definition}
\newtheorem{defn}[thm]{Definition}
\newtheorem{conj}[thm]{Conjecture}

\newtheorem*{crit*}{Criterion~B}

\definecolor{aquamarine}{rgb}{0.5, 1.0, 0.83}
\definecolor{princetonorange}{rgb}{1.0, 0.56, 0.0}
\definecolor{caribbeangreen}{rgb}{0.0, 0.8, 0.6}
\definecolor{bunired}{rgb}{0.8, 0.0, 0.0}
\definecolor{cdgreen}{rgb}{0.0, 0.42, 0.24}
\definecolor{lavender(floral)}{rgb}{0.71, 0.49, 0.86}
\definecolor{bluedefrance}{rgb}{0.19, 0.55, 0.91}
\definecolor{iris}{rgb}{0.35, 0.31, 0.81}
\definecolor{darkgreen}{rgb}{0.33, 0.42, 0.18}

\newcommand{\tG}{{\tt G}}

\newcommand{\tK}{{\tt K}}
\newcommand{\tI}{{\tt I}}

\newcommand{\tT}{{\tt T}}

\title{Multipath complexes of bidirectional polygonal digraphs}
\author{Luigi Caputi}
\author{Carlo Collari}
\author{Jason P. Smith}

\newcommand{\tP}{{\tt BP}}
\newcommand{\tL}{{\tt BL}}
\newcommand{\tW}{{\tt W}}
\newcommand{\taP}[1]{{\tt BP^{#1}_n}}

\newcommand{\tB}{{\tt B}}

\begin{document}

\maketitle

\begin{abstract}
In this work we study the homotopy type of multipath complexes of bidirectional path graphs and polygons, motivated by works of Vre\'cica and \v{Z}ivaljevi\'c on cycle-free chessboard
complexes (that is, multipath complexes of complete digraphs). In particular, we show that bidirectional path graphs are homotopic to spheres and that, in analogy with cycle-free chessboard complexes, multipath complexes of bidirectional polygonal digraphs are highly connected. Using a Mayer-Vietoris spectral sequence, we provide a computation of the associated homology groups. We study T-operations on graphs, and show that this corresponds to taking suspensions of multipath complexes.   
We further discuss (non) shellability properties of such complexes, and present new open questions.
\end{abstract}

\section*{Introduction}

A multipath in a directed graph is a collection of disjoint directed paths. The set of multipaths in a directed graph $G$ yields a simplicial complex $X(G)$, called the \emph{multipath complex}. 
Multipath complexes of complete directed graphs first appeared in the work of Vre\'cica and \v{Z}ivaljevi\'c~\cite{Omega}, under the name of \emph{cycle-free chessboard complexes}. Motivated by questions in homology theories of directed graphs~\cite{turner,zbMATH07965294}, multipath complexes have recently been the subject of investigations in~\cite{zbMATH07680365,zbMATH07796305,arXiv:2401.01248}. The
interest in studying the homology of multipath complexes in fact stems from their subtle relation to Hochschild homology~\cite{turner} and  symmetric homology of algebras~\cite{AultFed,AultNoFed}. In particular, Ault and Fiedorowicz proved that there is a spectral sequence, defined in terms of (suspensions of) the cycle-free chessboard complexes, 
converging to symmetric homology~\cite{AultFed}. 
The high connectivity of such complexes is essential to prove convergence properties and to compute the spectral sequence constructed by Ault and Fiedorowicz. In the case of cycle-free chessboard complexes, the required high connectivity was proven by Vre\'cica and \v{Z}ivaljevi\'c 
in~\cite[Theorem~10]{Omega}. 

In this work, motivated by the aforementioned results, we study the (high) connectivity properties of 
the multipath complexes of bidirectional 
polygonal digraphs $\tP_n$  (that is, of bidirectional cycles with $n+1$ vertices, see Figure~\ref{fig:polydash}). In particular, if we set  $\nu_n=\lfloor \frac{n-1}{2}\rfloor -1$, our main result is the following:
\begin{thmintro}[Theorem~\ref{thm:highconn}]
The multipath complex of the  digraph $\tP_{n}$  is $\nu_{n}$-connected.
\end{thmintro}

The first computations of the homotopy type of multipath complexes 
were provided in~\cite{zbMATH07796305} using techniques from combinatorial topology. In particular, it was shown that multipath complexes of oriented grids and transitive tournaments are contractible or homotopy equivalent to wedges of spheres. Multipath complexes of digraphs without directed cycles are also related to suitable matching complexes. In fact, it was shown in~\cite{arXiv:2401.01248}, that the multipath complexes  of digraphs without directed cycles are homotopy equivalent to matching complexes of bipartite (undirected) graphs -- see also Theorem~\ref{thm:multi=match}. This observation 
provides a strong relationship between multipath complexes, matching complexes and chessboard complexes.
A  consequence of this observation is that multipath complexes are not, in general, wedges of spheres as their homology may have torsion -- cf.~\cite{zbMATH07928744}. The main tool in establishing the connection between multipath complexes and matching complexes is  a blow-up operation of digraphs (see Definition~\ref{def:blowup}), and the absence of directed cycles is necessary to get homotopy equivalent complexes.  Based on these considerations, it is therefore  relevant to   analyse the homotopy type of multipath complexes in the case of directed graphs containing directed cycles. 

The easiest case of graphs with cycles, that is
oriented polygons, was investigated in~\cite{zbMATH07796305}. In this work,  we extend the study from oriented polygonal graphs to bidirectional polygonal digraphs. Further, we provide  computations of the homotopy type of multipath complexes of path graphs (that is, bidirectional linear graphs), 
and of digraphs obtained by T-operations -- see Proposition~\ref{prop:toper}. Furthermore, we study the shellability properties, and provide homology computations of their multipath complexes. In particular, using a Mayer-Vietoris spectral sequences argument, we show the following:
\begin{thmintro}[Theorem~\ref{thm:connectivity_homology}]
Let $X(\tP_n)$ be the multipath complex of the digraph $\tP_n$, for $n\geq 2$. Then, we have 
\[
\H_i(X(\tP_n)) = \begin{cases} 
\bZ  & \text{ if } i=0 \\
\bZ^2  & \text{ if } i=n-1 \\
\bZ  & \text{ if } i={\lfloor \frac{n - 1}{2} \rfloor} \text{ and } n\equiv 1,2 \mod 4\\
\bZ^3 & \text{if }   i={\lfloor \frac{n - 1}{2} \rfloor} \text{ and } n \equiv 3 \mod 4\\
\bZ  & \text{ if } i={\lfloor \frac{n - 1}{2} \rfloor}+1 \text{ and } n\equiv 0 \mod 4
\end{cases}.\]
In particular, $\tP_n$ is  not $(\nu_n+1)$-connected for $n\equiv 1,2,3 \mod 4$.  
\end{thmintro}

We remark  that we can not directly conclude from this computation that multipath complexes of bidirectional polygonal digraphs are, up to homotopy, wedges of spheres, and whether this is true is still open. Furthermore, we do not know if the given bound~$\nu_n$ from Theorem~\ref{thm:highconn} is sharp for 
$n\equiv 0 \mod 4$. This question is the analogue of a conjecture by Vre\'cica and \v{Z}ivaljevi\'c (see Conjecture~\ref{cong:VZ}) for bidirectional polygonal digraphs, and this also is still open.

\subsection*{Acknowledgments} LC was supported by the Starting Grant 101077154 “Definable Algebraic Topology” from the European Research Council of Martino Lupini. LC and CC are also grateful to INdAM-GNSAGA. CC partially acknowledges the
MUR-PRIN project 2022NMPLT8 and the MIUR Excellence Department Project awarded to the Department of Mathematics, University of Pisa, CUP I57G22000700001.

\section{Multipath complexes}
Henceforth, unless otherwise specified, by a digraph we shall mean a directed graph without multiple edges, that is between any pair of vertices there are at most two oriented edges going in opposite direction.

 Given an edge $e = (v, w)$ in the digraph $\tG$ we call the vertex $v$ the source of $e$, denoted $s(e)$, while the vertex $w$ is the
target of $e$, denoted $t(e)$. A  \emph{simple path} in $\tG$  is a sequence of edges $e_1,...,e_n$ of $\tG$ such that~$s(e_{i+1})=t(e_i)$ for $i=1,\dots,n-1$, and no vertex is encountered twice, i.e.~if $s(e_i) = s(e_j)$ or $t(e_i) = t(e_j)$, then $i=j$, and is not a cycle, i.e.~$s(e_1)\neq t(e_n)$ -- cf.~Figure~\ref{fig:nstep}. 

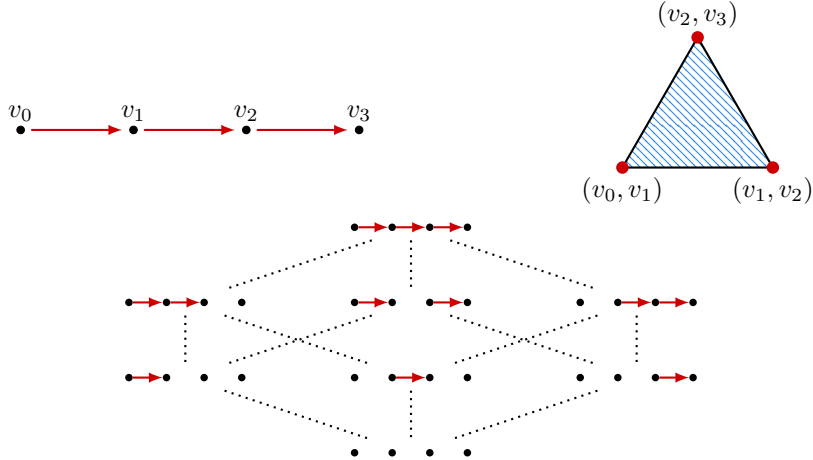
\begin{figure}[h]
\centering
	\begin{tikzpicture}[baseline=(current bounding box.center),line join = round, line cap = round]
		\tikzstyle{point}=[circle,thick,draw=black,fill=black,inner sep=0pt,minimum width=2pt,minimum height=2pt]
		\tikzstyle{arc}=[shorten >= 8pt,shorten <= 8pt,->, thick]
		\def\c{8}\def\d{.5}
		\node[above] (v0) at (0-\c,0+\d) {$v_0$};\draw[fill] (0-\c,0+\d)  circle (.05);
		\node[above] (v1) at (1.5-\c,0+\d) {$v_1$};\draw[fill] (1.5-\c,0+\d)  circle (.05);
		\node[above] (v2) at (3-\c,0+\d) {$v_{2}$};\draw[fill] (3-\c,0+\d)  circle (.05);
		\node[above] (v3) at (4.5-\c,0+\d) {$v_{3}$};\draw[fill] (4.5-\c,0+\d)  circle (.05);
		
		\draw[thick, bunired, -latex] (0.15-\c,0+\d) -- (1.35-\c,0+\d);
		\draw[thick, bunired, -latex] (1.65-\c,0+\d) -- (2.85-\c,0+\d);
		\draw[thick, bunired, -latex] (3.15-\c,0+\d) -- (4.35-\c,0+\d);
		
		\node (e1) at (0,0) {};
		\node (e2) at (2,0) {};
		\node (e3) at (60:2) {};
		\draw[pattern=north west lines, pattern color=bluedefrance,thick] (e1.center) -- (e2.center) -- (e3.center) -- (e1.center);
		\node[circle,fill=bunired,scale=0.5] at (e1) {};\node[below] at (e1) {$(v_0,v_1)$};
        \node[circle,fill=bunired,scale=0.5] at (e2) {};\node[below] at (e2) {$(v_1,v_2)$};
        \node[circle,fill=bunired,scale=0.5] at (e3) {};\node[above] at (e3) {$(v_2,v_3)$};
    \end{tikzpicture}
	\begin{tikzpicture}
        \def\y{1}\def\x{3}
        \tikzstyle{sp}=[circle,fill=black,scale=0.3]
        \tikzstyle{se}=[thick, bunired, -latex]
        \node (p123) at (0*\x,3*\y){\begin{tikzpicture}[scale=0.5]\node[sp] (0) at (0,0){};\node[sp] (1) at (1,0){};\node[sp] (2) at (2,0){};\node[sp] (3) at (3,0){};\draw[se] (0) -- (1);\draw[se] (1) -- (2);\draw[se] (2) -- (3);\end{tikzpicture}};
        
        \node (p12) at (-1*\x,2*\y){\begin{tikzpicture}[scale=0.5]\node[sp] (0) at (0,0){};\node[sp] (1) at (1,0){};\node[sp] (2) at (2,0){};\node[sp] (3) at (3,0){};\draw[se] (0) -- (1);\draw[se] (1) -- (2);\end{tikzpicture}};
        \node (p13) at (0*\x,2*\y){\begin{tikzpicture}[scale=0.5]\node[sp] (0) at (0,0){};\node[sp] (1) at (1,0){};\node[sp] (2) at (2,0){};\node[sp] (3) at (3,0){};\draw[se] (0) -- (1);\draw[se] (2) -- (3);\end{tikzpicture}};
        \node (p23) at (1*\x,2*\y){\begin{tikzpicture}[scale=0.5]\node[sp] (0) at (0,0){};\node[sp] (1) at (1,0){};\node[sp] (2) at (2,0){};\node[sp] (3) at (3,0){};\draw[se] (1) -- (2); \draw[se] (2) -- (3);\end{tikzpicture}};
        \node (p1) at (-1*\x,1*\y){\begin{tikzpicture}[scale=0.5]\node[sp] (0) at (0,0){};\node[sp] (1) at (1,0){};\node[sp] (2) at (2,0){};\node[sp] (3) at (3,0){};\draw[se] (0) -- (1);\end{tikzpicture}};
        \node (p2) at (0*\x,1*\y){\begin{tikzpicture}[scale=0.5]\node[sp] (0) at (0,0){};\node[sp] (1) at (1,0){};\node[sp] (2) at (2,0){};\node[sp] (3) at (3,0){};\draw[se] (1) -- (2);\end{tikzpicture}};
        \node (p3) at (1*\x,1*\y){\begin{tikzpicture}[scale=0.5]\node[sp] (0) at (0,0){};\node[sp] (1) at (1,0){};\node[sp] (2) at (2,0){};\node[sp] (3) at (3,0){};\draw[se] (2) -- (3);\end{tikzpicture}};
        \node (p0) at (0*\x,0*\y){\begin{tikzpicture}[scale=0.5]\node[sp] (0) at (0,0){};\node[sp] (1) at (1,0){};\node[sp] (2) at (2,0){};\node[sp] (3) at (3,0){};\end{tikzpicture}};
        \draw[thick, dotted] (p123) -- (p12) -- (p1) -- (p0);
        \draw[thick,dotted] (p123) -- (p23) -- (p2) -- (p0);
        \draw[thick,dotted] (p123) -- (p13) -- (p3) -- (p0);
        \draw[thick,dotted] (p12) -- (p2);
        \draw[thick,dotted] (p23) -- (p3);
        \draw[thick,dotted] (p13) -- (p1);
	\end{tikzpicture}

	\caption{The coherently oriented linear graph $\tI_3$ (top left), the multipath complex $X(\tI_3)$ (top right), and the path poset $P(\tI_3)$ (bottom).}
	\label{fig:nstep}
\end{figure}

We are interested in disjoint sets of simple paths; following~\cite{turner, zbMATH07965294, zbMATH07680365}, we call them multipaths:

\begin{defn}\label{def:multipaths}
A \emph{multipath} of a digraph~$\tG$ is a spanning subgraph such that each connected component is either a vertex or a simple path. The \emph{length} of a multipath is the number of edges. 
\end{defn}

The set of multipaths of $\tG$ has a natural partially ordered structure: the \emph{path poset} of $\tG$ is the poset $(P(\tG),<)$, that is, the set of multipaths of $\tG$ (including the multipath with no edges) ordered by the relation of ``being a subgraph''. 
To the path poset we can associate a simplicial complex, which we call the multipath complex  -- cf.~\cite[Definition~6.4]{zbMATH07680365}:

\begin{defn}\label{def:pathcomplx}
For a digraph $\tG$, the \emph{multipath complex} $X(\tG)$ is the simplicial complex whose face poset (augmented to include the empty simplex~$\emptyset$) is the path poset $P(\tG)$. 
\end{defn}

Since being a multipath is a monotone property of digraphs, it follows that $X(\tG)$ is a well-defined simplicial complex. The study of the homotopy type of multipath complexes was initiated in~\cite{zbMATH07796305}, where it was shown that multipath complexes of polygonal graphs are homotopy equivalent to wedges of spheres. The aim of this paper is to explore the structure of bidirectional polygonal graphs. Similarly, the multipath complex of  transitive tournaments is homotopy equivalent to a wedge of spheres by \cite[Theorem~5.1]{zbMATH07796305}. 

Recall that the \emph{indegree} (resp. \emph{outdegree}) of a vertex $v$ is the number of edges with target (resp.~source)~$v$.

\begin{defn}[{\cite[Definition~4.15]{arXiv:2401.01248}}]\label{def:blowup}
Let $\tG$ be a digraph, and let $v\in V(\tG)$ with both indegree and outdegree different from $0$. The \emph{blow-up} of $\tG$ at $v$ is the digraph $B(\tG,v)$ obtained from $\tG$ as follows: the vertices of $B(\tG,v)$ are the same as $\tG$ except for $v$, which is replaced by $v_{\mathrm{in}}$ and $v_{\mathrm{out}}$. Given $v',v''\in V(B(\tG,v))$ the edges from $v'$ to $v''$ are in bijection with 
\begin{enumerate}[label = (\alph*)]
\item the edges between the corresponding vertices in $\tG$, if $\{v',v''\}\cap \{v_{\mathrm{in}},v_{\mathrm{out}}\} = \emptyset$;
\item the edges from $v'$ to $v$, if $v'\neq v_{\mathrm{out}}$ and $v''=v_{\mathrm{in}}$;
\item the edges from $v$ to $v''$, if $v'=v_{\mathrm{out}}$ and $v''\neq v_{\mathrm{in}}$;
\item the empty set, in the remaining cases.
\end{enumerate}
If the indegree or the outdegree of $v$ are zero, then we set $B(\tG,v) \coloneqq \tG$.
The \emph{blow-up} of $\tG$ is the digraph $B(\tG)$ obtained from $\tG$ by iteratively blowing-up all vertices in $\tG$ one after the other; see Figure~\ref{fig:blowup} for an illustrative example.
\end{defn}

Note that the blow-up of digraphs does not depend on the order in which we blow-up the vertices. Furthermore, the underlying undirected graph of the blow-up of a digraph is always a bipartite graph.

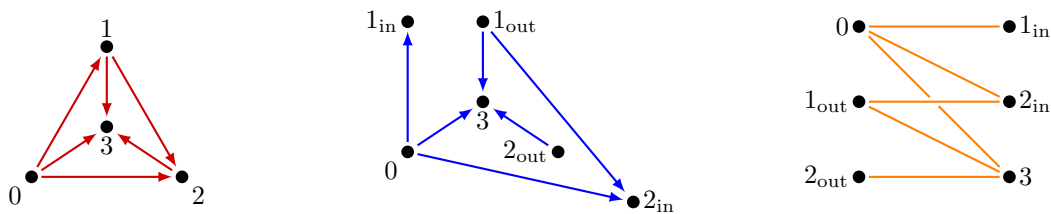
\begin{figure}[h]
    \centering
    \begin{tikzpicture}[scale = 2, thick]
    \node (a) at (0,0) {};
    \node (b) at (.5,.866) {};
    \node (c) at (1,0) {};
    \node (d) at (.5,.333) {};
    
    \node at (0,0) [below left] {$ 0$};
    \node at (.5,.866) [above] {$ 1$};
    \node at (1,0) [below right] {$ 2$};
    \node at (.5,.333) [below] {$3$};

    \draw[black, fill] (a) circle (.035);
    \draw[black, fill] (b) circle (.035);
    \draw[black, fill] (c) circle (.035);
    \draw[black, fill] (d) circle (.035);

    \draw[-latex, bunired] (a) -- (b);
    \draw[-latex, bunired] (a) -- (c);
    \draw[-latex, bunired] (b) -- (c);
    \draw[-latex, bunired] (a) -- (d);
    \draw[-latex, bunired] (b) -- (d);
    \draw[-latex, bunired] (c) -- (d);
    
    \begin{scope}[shift = {+(2.5,0.166)}]

    \node (a) at (0,0) {};
    \node (bout) at (.5,.866) {};
    \node (bin) at (0,.866) {};
    \node (cout) at (1,0) {};
    \node (cin) at (1.5,-.333) {};
    \node (d) at (.5,.333) {};
    
    \node at (0,0) [below left] {$ 0$};
    \node at (.5,.866) [right] {$1_{\mathrm{out}}$};
    \node at (1,0) [left] {$2_{\mathrm{out}}$};
    \node at (0,.866) [left] {$1_{\mathrm{in}}$};
    \node at (1.5,-.333) [ right] {$2_{\mathrm{in}}$};
    \node at (.5,.333) [below] {$3$};

    \draw[black, fill] (a) circle (.035);
    \draw[black, fill] (bin) circle (.035);
    \draw[black, fill] (cin) circle (.035);
    \draw[black, fill] (bout) circle (.035);
    \draw[black, fill] (cout) circle (.035);
    \draw[black, fill] (d) circle (.035);

    \draw[-latex, blue] (a) -- (bin);
    \draw[-latex, blue] (a) -- (cin);
    \draw[-latex, blue] (bout) -- (cin);
    \draw[-latex, blue] (a) -- (d);
    \draw[-latex, blue] (bout) -- (d);
    \draw[-latex, blue] (cout) -- (d);
    \end{scope}

        \begin{scope}[shift = {+(5.5,0)}]

    \node (a) at (0,1) {};
    \node (b) at (0,0.5) {};
    \node (c) at (0,0) {};
    \node (ap) at (1,1) {};
    \node (bp) at (1,0.5) {};
    \node (cp) at (1,0) {};
    
    \node at (0,1) [left] {$ 0$};
    \node at (0,0.5) [left] {$1_{\mathrm{out}}$};
    \node at (0,0) [left] {$2_{\mathrm{out}}$};
    \node at (1,1) [right] {$1_{\mathrm{in}}$};
    \node at (1,0.5) [right] {$2_{\mathrm{in}}$};
    \node at (1,0) [right] {$3$};

    \draw[black, fill] (a) circle (.035);
    \draw[black, fill] (b) circle (.035);
    \draw[black, fill] (c) circle (.035);
    \draw[black, fill] (ap) circle (.035);
    \draw[black, fill] (bp) circle (.035);
    \draw[black, fill] (cp) circle (.035);

    \draw[orange] (a) -- (ap);
    \draw[orange] (a) -- (bp);
    \draw[orange] (a) -- (cp);
        \draw[white, fill] (0.5,0.5) circle (.035);
    \draw[orange] (b) -- (bp);
    \draw[orange] (b) -- (cp);
    \draw[orange] (c) -- (cp);
    \end{scope}
    
\end{tikzpicture}
\caption{Transitive tournament $\tT_3$ (in red on the left), its blow-up (in blue at the centre), and the associated bipartite graph $\tB_3$ (in orange on the right).}
\label{fig:blowup}
\end{figure}

Denote by $\iota(\tG)$ the undirected graph underlying the digraph $\tG$. Then, if $\tG$ has no oriented cycles, we have a close relation of multipath complexes to matching complexes: 

\begin{thm}[{\cite[Theorem~4.20]{arXiv:2401.01248}}]\label{thm:multi=match}
Let $\tG$ be a digraph without oriented cycles. Then, the multipath complex~$X(\tG)$ is isomorphic, as a simplicial complex, to the matching complex of $\iota(B(\tG))$. 
\end{thm}

As a consequence, multipath complexes can have torsion -- see also~\cite[Proposition~4.5]{zbMATH07928744}. In view of Theorem~\ref{thm:multi=match}, the homotopy type of multipath complexes of digraphs without oriented cycles can be computed as the homotopy type of matching complexes of bipartite graphs. It is yet unclear what is the relation in the presence of oriented cycles. 

\subsection{ Cycle-free chessboard complexes}

Recall that an $(m\times n)$-chessboard is the set $A_{m,n}=[m]\times [n]\subseteq \mathbb{Z}^2$, where we denote by $[n]$ the set $\{1,\dots,n\}$. 

A chessboard $A_n=[n]\times[n]$ can be interpreted as a complete digraph $\tK_n$, that is the complete digraph on $n$ vertices and edges $i\to j$ for all $i,j\in [n]$. Clearly, each $(i,j)\in A_n$ corresponds to an arrow $i\to j$ in~$\tK_n$.  Then, the \emph{chessboard complex}~$\Delta_n$ is the  complex of all subgraphs of $\tK_n$ such that the associated connected
components are either directed cycles or directed paths. 

\begin{defn}[\cite{Omega}]
    The \emph{cycle-free chessboard complex}~$\Omega_n$ is the subcomplex of $\Delta_n$ spanned by cycle-free digraphs.
\end{defn}

More precisely, the cycle-free chessboard complex~$\Omega_n$ is the complex of all subgraphs of $\tK_n$ such that the associated connected
components are  directed paths. Therefore, $\Omega_n$ is the multipath complex $X(\tK_n)$. 

By \cite[Proposition~9]{Omega}, the cycle-free chessboard complexes are simply connected for all $n\geq 5$. More generally, in \cite[Section~7]{Omega} it was shown that~$\Omega_n$ is highly connected, that is:

\begin{thm}[{\cite[Theorem~10]{Omega}}]
    The cycle-free chessboard complex~$\Omega_n$ is $\mu_n$-connected, where \[\mu_n=\left\lfloor \frac{2n-1}{3}-2\right\rfloor.\]
\end{thm}

The constant $\mu_n$ is the best possible for $n=3k+2$~\cite[Section~8]{Omega}. However, whether $\mu_n$ is always the best possible connectivity bound it still is an open conjecture:

\begin{conj}[\cite{Omega}]\label{cong:VZ}
    The connectivity bound $\mu_n$ is the best possible, that is, we have 
    \[
    H_{\mu_{n}+1}(\Omega_n)=H_{\mu_{n}+1}(X(\tK_n))\neq 0 
    \]
     for each $n\geq 2$.
\end{conj}

\section{Bidirectional polygonal digraphs}

The aim of this section is to provide a connectivity bound for multipath complexes of bidirectional linear graphs and bidirectional polygonal graphs (i.e. bidirectional cycle graphs), and explicit homology computations of both. Recall first the following result, allowing us to deduce the homotopy type of a multipath complex from the homotopy type of its pieces.

\begin{lem}[{\cite[Lemma~10.4(ii)]{BjorTopMeth}}]\label{lem:BjornerLemma}
Suppose that $X$ is a simplicial complex which can be written as the union of subcomplexes $X_{0},\dots, X_{n}$ such that:
\begin{enumerate}[label = (\alph*)]
\item $X_i$ is contractible for each $i= 0,\dots, n$, and
\item $X_i\cap X_j \subseteq X_0$ for all $i,j\in \{ 1,..,n\}$ and $i\neq j$.
\end{enumerate}
There is a homotopy equivalence
$ X \simeq \bigvee_{i=1}^n \Sigma(X_{0} \cap X_{i}) $,
where $\Sigma(Y)$ denotes the topological suspension of~$Y$.
\end{lem}

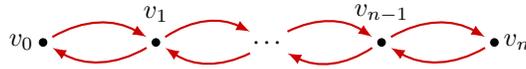
\begin{figure}[h]
\begin{tikzpicture}[baseline=(current bounding box.center)]
		\tikzstyle{point}=[circle,thick,draw=black,fill=black,inner sep=0pt,minimum width=2pt,minimum height=2pt]
		\tikzstyle{arc}=[shorten >= 8pt,shorten <= 8pt,->, thick]
		
		\node (v0) at (0,0) {};
		\node[left] at (0,0) {$v_0$};
		\node[above] at (1.5,0.15) {$v_1$};
		\node[above] at (4.5,0.15) {$v_{n-1}$};
		\node[right] at (6,0) {$v_n$};
		\draw[fill] (0,0)  circle (.05);
		\node (v1) at (1.5,0) {};
		\draw[fill] (1.5,0)  circle (.05);
		\node (v2) at (3,0) {\dots};
		\node (v4) at (4.5,0) {};
		\draw[fill] (4.5,0)  circle (.05);
		\node (v5) at (6,0) {};
		\draw[fill] (6,0)  circle (.05);
		
		\draw[thick, bunired, -latex] (v0) to[bend left] (v1);
  		\draw[thick, bunired, -latex] (v1) to[bend left] (v0);
\draw[thick, bunired, -latex] (v1) to[bend left] (v2);
  		\draw[thick, bunired, -latex] (v2) to[bend left] (v1);

\draw[thick, bunired, -latex] (v2) to[bend left] (v4);
  		\draw[thick, bunired, -latex] (v4) to[bend left] (v2);
\draw[thick, bunired, -latex] (v5) to[bend left] (v4);
  		\draw[thick, bunired, -latex] (v4) to[bend left] (v5);
\end{tikzpicture}
	\caption{The bidirectional linear graph $\tL_n$.}
	\label{fig:nbidlin}
\end{figure}

For $n\geq 0$, let $\tL_n$ be the bidirectional linear digraph with $2n$ edges and $(n+1)$-vertices, as depicted in Figure~\ref{fig:nbidlin}. In particular, the graph $\tL_0$ is a single vertex with no edges. Using Lemma~\ref{lem:BjornerLemma}, we can determine the homotopy type of $X(\tL_n)$:

\begin{prop}\label{prop:connectivitylin}
    The  multipath complex $X(\tL_n)$ is homotopy equivalent to the sphere $S^{\lfloor \frac{n - 1}{2} \rfloor}$.
\end{prop}

\begin{proof}
Consider $X_0, X_1 \subseteq X$ defined as follows:
\[ X_{0} \coloneqq \{ \text{ all multipaths containing or composable with }(v_0,v_1) \} \]
and
\[ X_{1} \coloneqq \{ \text{ all multipaths containing or composable with }(v_1,v_0) \}. \]
Clearly $X = X_{0} \cup X_1$. The complexes $X_0$ and $ X_1 $ are contractible, as they are cones.  Moreover, $X_0 \cap X_1 \cong X(\tL_{n-2})$.
It follows from Lemma~\ref{lem:BjornerLemma} that $X(\tL_n)$ is homotopy equivalent to the suspension $  \Sigma X(\tL_{n-2})$.
The multipath complex of $\tL_0$ is empty, and it can be easily checked that $X(\tL_{1})  \simeq X(\tL_{2})  \simeq S^0$. The statement follows. 
\end{proof}

\begin{figure}[h]
\begin{tikzpicture}[baseline=(current bounding box.center), scale=0.7]
		\tikzstyle{point}=[circle,thick,draw=black,fill=black,inner sep=0pt,minimum width=2pt,minimum height=2pt]
		\tikzstyle{arc}=[shorten >= 8pt,shorten <= 8pt,->, thick]
		
		\node[left] at (0,0) {$v_0$};
		\node (v0) at (0,0) {};
		\draw[fill] (0,0)  circle (.05);
		\node (v1) at (1.5,0) {};
		\draw[fill] (1.5,0)  circle (.05);
		\node (v2) at (3,0) {\dots};
		\node (v4) at (4.5,0) {};
		\draw[fill] (4.5,0)  circle (.05);
		\node (v5) at (6,0) {};
  \node[above] at (6.2,0) {$v_n$};
		\draw[fill] (6,0)  circle (.05);
  \node[right] (v6) at (7.5,0) {$v_{n+1}$};
		\draw[fill] (7.5,0)  circle (.05);
		
		\draw[thick, bunired, -latex] (v0) to (v1);
\draw[thick, bunired, -latex] (v1) to[bend left] (v2);
  		\draw[thick, bunired, -latex] (v2) to[bend left] (v1);

\draw[thick, bunired, -latex] (v2) to[bend left] (v4);
  		\draw[thick, bunired, -latex] (v4) to[bend left] (v2);
\draw[thick, bunired, -latex] (v5) to[bend left] (v4);
  		\draw[thick, bunired, -latex] (v4) to[bend left] (v5);
    \draw[thick, bunired, -latex] (v5) to (7.4,0);
\end{tikzpicture}
\hfill
\begin{tikzpicture}[baseline=(current bounding box.center), scale=0.7]
		\tikzstyle{point}=[circle,thick,draw=black,fill=black,inner sep=0pt,minimum width=2pt,minimum height=2pt]
		\tikzstyle{arc}=[shorten >= 8pt,shorten <= 8pt,->, thick]
		
		\node[left] at (0,0) {$v_0$};
		\node (v0) at (0,0) {};
		\draw[fill] (0,0)  circle (.05);
		\node (v1) at (1.5,0) {};
		\draw[fill] (1.5,0)  circle (.05);
		\node (v2) at (3,0) {\dots};
		\node (v4) at (4.5,0) {};
		\draw[fill] (4.5,0)  circle (.05);
		\node (v5) at (6,0) {};
  \node[above] at (6.2,0) {$v_n$};
		\draw[fill] (6,0)  circle (.05);
  \node[right] (v6) at (7.5,0) {$v_{n+1}$};
		\draw[fill] (7.5,0)  circle (.05);
		
		\draw[thick, bunired, -latex] (v0) to[bend left] (v1);
  		\draw[thick, bunired, -latex] (v1) to[bend left] (v0);
\draw[thick, bunired, -latex] (v1) to[bend left] (v2);
  		\draw[thick, bunired, -latex] (v2) to[bend left] (v1);

\draw[thick, bunired, -latex] (v2) to[bend left] (v4);
  		\draw[thick, bunired, -latex] (v4) to[bend left] (v2);
\draw[thick, bunired, -latex] (v5) to[bend left] (v4);
  		\draw[thick, bunired, -latex] (v4) to[bend left] (v5);
    \draw[thick, bunired, -latex] (v5) to (7.4,0);
\end{tikzpicture}

	\caption{The  linear graph $\tW_n$ (on the left) and the graph $\widehat{\tL_n}$ (on the right).}
	\label{fig:nbidlinplusone}
\end{figure}
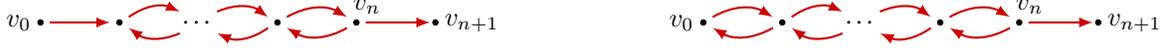

Inspection of the proof of Proposition~\ref{prop:connectivitylin} provides an explicit sphere~$S^{\lfloor \frac{n - 1}{2} \rfloor}$ generating the homology groups of $X(\tL_n)$.
In the notation used in the proof of Proposition~\ref{prop:connectivitylin}, note that $X_0$ is homotopy equivalent to a cone of simplicial complexes, \emph{i.e.}~it is the cone~$(v_0,v_1)\ast X(\tL_{n-2})$ with apex $(v_0,v_1)$, where we have identified $\tL_{n-2}$ with the full subgraph of $\tL_n$ on the vertices $v_2,\dots,v_n$. The homotopy equivalence can be obtained via a deformation retraction which fixes $X_0\cap X_1$. Analogously, $X_1$ is also homotopy equivalent to the cone $(v_1,v_0)\ast X(\tL_{n-2})$ via a deformation which fixes $X_0\cap X_1$. Therefore, the multipath complex $X(\tL_n)$ is obtained, up to homotopy,  by taking the suspension of $X(\tL_{n-2})=X_0\cap X_1$ with suspension points corresponding to $(v_0,v_1)$ and $(v_1,v_0)$.
In view of these identifications, we get that the multipath complex~$X(\tL_n)$ is homotopy equivalent to the join:
\begin{equation}
    \label{eq:almostbid}
    X'_n = \begin{cases} \left\lbrace(v_0,v_1) , (v_1,v_0)\right\rbrace \ast \left\lbrace (v_2,v_3), (v_3,v_2) \right\rbrace \ast \cdots \ast \left\lbrace (v_{n-1},v_n) , (v_n,v_{n-1})\right\rbrace & n \equiv 1 \mod 2 \\ \\ \left\lbrace  (v_0,v_1) , (v_1,v_0)\right\rbrace \ast \left\lbrace  (v_2,v_3) , (v_3,v_2) \right\rbrace \ast \cdots \ast \left\lbrace  (v_{n-2},v_{n-1}) , (v_{n-1},v_{n-2})\right\rbrace & n \equiv 0 \mod 2 \end{cases}.\
\end{equation}
In either case the triples $(X(\tL_n); X_0, X_1)$ and $(X'_n; (v_0,v_1)\ast X'_{n-1}, (v_1,v_0)\ast X'_{n-1})$ are Mayer-Vietoris triads in the sense of \cite[Page~240]{Brown68}.
By induction one shows that the natural inclusions induce a map of triads 
\[(X'_n; (v_0,v_1)\ast X'_{n-1}, (v_1,v_0)\ast X'_{n-1})\longrightarrow (X(\tL_n); X_0, X_1)\ .\]
Thus, the inclusion $X'\hookrightarrow X(\tL_n)$ is a homotopy equivalence by  \cite[7.4.1]{Brown68}. Similarly, one can show that $X(\tL_{4k})$ is homotopy equivalent to
\begin{equation}
    \label{eq:almostbid2}
    X''_{2k} = \left\lbrace(v_0,v_1) , (v_2,v_1)\right\rbrace \ast \left\lbrace (v_3,v_2), (v_3,v_4) \right\rbrace \ast \cdots \ast \left\lbrace (v_{4k-1},v_{4k - 2}) , (v_{4k-1},v_{4k})\right\rbrace
\end{equation}
with inclusion $X''_{2k}\hookrightarrow X(\tL_{2k})$ giving a homotopy equivalence.

 Lemma~\ref{lem:BjornerLemma} can be of help in a variety of cases, as we shall explain. Consider the case of graphs $\tG'$ as illustrated in Figure~\ref{fig:gluingt}. Assume that in $\tG'$ there is a vertex $v$ of valence 3 without selfloops. Assume also that the indegree or outdegree of $v$ is $1$. Let $\tW_n$ be the graph obtained from $\tL_n$ by deleting the edge $(v_1,v_0)$ and adding the edge~$(v_n,v_{n+1})$, as depicted in Figure~\ref{fig:nbidlinplusone}. Then, we can glue a copy of $\tW_2$ to $\tG'$,   by identifying $v$ with either $v_0$ or $v_{n+1}$ and $t$ with either $v_{1}$ or $v_{n}$, depending on the orientation of the edge between $t$ and $v$. The result of such an operation is shown in Figure~\ref{fig:gluingt}. We say that the new graph $\tG''$ obtained by gluing of $\tW_2$ on $\tG'$ as described has been obtained from $\tG'$ by a \emph{T-operation}.

\begin{prop}\label{prop:toper}
    Let $\tG''$ be a graph obtained from $\tG'$ by a T-operation. Then, we have 
    $
    X(\tG'')\simeq \Sigma X(\tG')$.
    {Furthermore, if we identify $\Sigma X(\tG')$ with the subcomplex of $
    X(\tG'')$ given by $\{(t_2,t_1), (t, t_1)\}\ast X(\tG')$, then the natural inclusion map~$\Sigma X(\tG') \hookrightarrow X(\tG'')$ is a homotopy equivalence.}
\end{prop}

\begin{proof}
    We assume that $v$, the vertex of valence 3, has indegree $1$ -- the other case can be proven similarly. We use the same notation illustrated in Figure~\ref{fig:gluingt}. Let $X_0$ be the subcomplex of $X(\tG'')$ given by all multipaths containing or composable with $(t_2,t_1)$. Analogously, let $X_1$ be the subcomplex of $X(\tG'')$ given by all multipaths containing or composable with $(t,t_1)$. Then, they cover $X(\tG'')$ and their intersection $X_0\cap X_1$ can be identified with $X(\tG')$. The first part of the statement follows from Lemma~\ref{lem:BjornerLemma}.
    {The last part of the statement follows from  \cite[7.4.1]{Brown68} since $(X(\tG'' ), X_0, X_1)$ and $(X(\tG'), (t_2,t_1)\ast X(\tG'),  (t, t_1) \ast X(\tG'))$ are Mayer-Vietoris triples, and the inclusions give a map of triples.}
\end{proof}

{Note that, in the proof of Proposition~\ref{prop:toper} we have never used the fact that $v$ is of valence 3,  and we have never used the structure of $\tG$. In fact, the proposition can be generalised to graphs $\tG'$ in which there exists a vertex $t$ which has either indegree $0$ or outdegree $0$. Let $e_1,\dots, e_k$ be the edges in $\tG'$ incident to $t$, and call $u_0, \dots, u_s$  the set of vertices incident to $e_1,\dots, e_k$. Then, let $\tG''$ be obtained from $\tG'$ by gluing a copy of $\tW_2$ (on vertices $v_0,v_1,v_2,v_3$). If $t$ has indegree 0, we identify $t$ and $v_1$, and $v_0$ with $u_j$ for some $j$; otherwise, we identify $t$ with $v_2$, and $v_3$ with  $u_j$ for some~$j$. Then, with the same arguments of Proposition~\ref{prop:toper}, we get $X(\tG'')\simeq \Sigma X(\tG')$.}

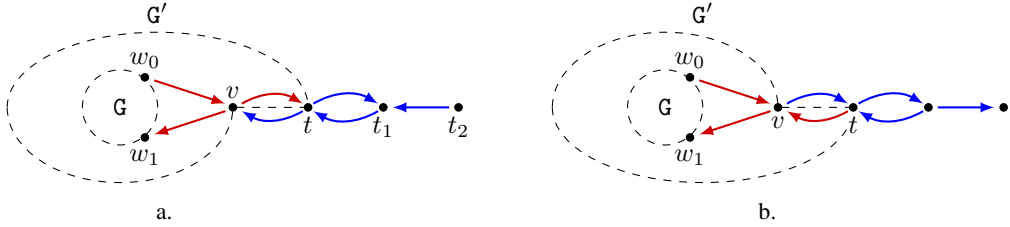
\begin{figure}[h]
	\centering
	\begin{subfigure}[b]{0.3\textwidth}
		\centering
\begin{tikzpicture}[baseline=(current bounding box.center)]
		\tikzstyle{point}=[circle,thick,draw=black,fill=black,inner sep=0pt,minimum width=2pt,minimum height=2pt]
		\tikzstyle{arc}=[shorten >= 8pt,shorten <= 8pt,->, thick]
        \node[above] at (-.5,1) {$\tt{G'}$};
        \node at (-1,0) {$\tt{G}$};
		\draw[dashed] (-1,0)  circle (.5);
  \draw[dashed] (-2.5,0)  arc (180:0:2 and 1) -- (.5,0)  arc (0:-180:1.5 and 1);
		\node[above] (v0) at (0.5,0) {$v$};
		\draw[fill] (0.5,0)  circle (.05);

\node[above]  at (-0.67,0.4) {$w_0$};
\node (w0) at (-0.67,0.4) {};
		\draw[fill] (-0.67,0.4)  circle (.05);
  \node[below]  at (-0.67,-0.44) {$w_1$};
  \node (w1) at (-0.67,-0.4) {};
		\draw[fill] (-0.67,-0.4)  circle (.05);
  
		\node[fill, white, below] at (1.5,0) {}  circle (.25);
		\node[below] (v1) at (1.5,0) {$t$};
		\draw[fill] (1.5,0)  circle (.05);
		\node (v2) at (2.5,0) {};
		\node[below] at (3.5,0) {$t_{2}$};
		\node[below] at (2.5,0) {$t_{1}$};
		\draw[fill] (2.5,0)  circle (.05);
		\node (v3) at (3.5,0) {};
		\draw[fill] (3.5,0)  circle (.05);

\draw[thick, bunired, -latex] (w0) -- (0.4,0.05);
\draw[thick, bunired, latex-] (w1) -- (0.4,-0.05);
  
		\draw[thick, bunired, latex-] (1.43,0.05) to[bend right] (0.62,0.05);
  \draw[thick, blue, -latex] (1.43,-0.05) to[bend left] (0.62,-0.05);

  \draw[thick, blue, latex-] (1.57,-0.05) to[bend right] (2.43,-0.05);
  \draw[thick, blue, -latex] (1.57,0.05) to[bend left] (2.43,0.05);
  \draw[thick, blue, latex-] (v2) -- (v3);
  
	\end{tikzpicture}
		\caption{\phantom{  A   }}
	\end{subfigure}
	\hspace{0.1\textwidth}
	\begin{subfigure}[b]{0.5\textwidth}
	\centering
\begin{tikzpicture}[baseline=(current bounding box.center)]
		\tikzstyle{point}=[circle,thick,draw=black,fill=black,inner sep=0pt,minimum width=2pt,minimum height=2pt]
		\tikzstyle{arc}=[shorten >= 8pt,shorten <= 8pt,->, thick]
        \node[above] at (-.5,1) {$\tt{G'}$};
        \node at (-1,0) {$\tt{G}$};
		\draw[dashed] (-1,0)  circle (.5);
		
		\node[below] (v0) at (0.5,0) {$v$};
		\draw[fill] (0.5,0)  circle (.05);
  \draw[dashed] (-2.5,0)  arc (-180:0:2 and 1) -- (.5,0)  arc (0:180:1.5 and 1);
\node[above]  at (-0.67,0.4) {$w_0$};
\node (w0) at (-0.67,0.4) {};
		\draw[fill] (-0.67,0.4)  circle (.05);
  \node[below]  at (-0.67,-0.44) {$w_1$};
  \node (w1) at (-0.67,-0.4) {};
		\draw[fill] (-0.67,-0.4)  circle (.05);
  
		\node[fill, white, below] at (1.5,0) {}  circle (.25);
		\node[below] (v1) at (1.5,0) {$t$};
		\draw[fill] (1.5,0)  circle (.05);
		\node (v2) at (2.5,0) {};
		\draw[fill] (2.5,0)  circle (.05);
		\node (v3) at (3.5,0) {};
		\draw[fill] (3.5,0)  circle (.05);

\draw[thick, bunired, -latex] (w0) -- (0.4,0.05);
\draw[thick, bunired, latex-] (w1) -- (0.4,-0.05);
  
		\draw[thick, blue, latex-] (1.43,0.05) to[bend right] (0.62,0.05);
  \draw[thick, bunired, -latex] (1.43,-0.05) to[bend left] (0.62,-0.05);

  \draw[thick, blue, latex-] (1.57,-0.05) to[bend right] (2.43,-0.05);
  \draw[thick, blue, -latex] (1.57,0.05) to[bend left] (2.43,0.05);
  \draw[thick, blue, -latex] (v2) -- (v3);
  
	\end{tikzpicture}
				\caption{ \phantom{C }}
\end{subfigure}
\caption{The definition of $T$-operation, which is a glueing of a copy of $\tW_2$ (in blue) to a graph $\tG'$.}
\label{fig:gluingt}
\end{figure}

Denote by $\widehat{\tL_n}$ the graph obtained from $\tL_n$ by adding an extra edge $(v_n,v_{n+1})$, as depicted in Figure~\ref{fig:nbidlinplusone}. In the light of Proposition~\ref{prop:toper}, we can now compute the homotopy type of the multipath complex $X(\widehat{\tL_n})$;

\begin{cor}\label{cor:BLnhat}
For each $n$ we have a homotopy equivalence
\[ X(\widehat{\tL_n}) \simeq \begin{cases} \ast & \text{if } n \text{ is even} \\ S^{ \frac{n - 1}{2}} &\text{if } n \text{ is odd} \end{cases}.\]
\end{cor}
\begin{proof} By direct computation, it is easy to verify that the statement is true for $n = 0,1$. 
Assume the statement is true for all $n< k$, and denote by $-\widehat{\tL_{n}}$ the digraph obtained by reversing the orientation of all edges in $\widehat{\tL_{n}}$. Directly from Proposition~\ref{prop:toper}, we have the homotopy equivalence $X(\widehat{\tL_k}) \simeq \Sigma X(-\widehat\tL_{k-2})$, which in turn is isomorphic to~$\Sigma X(\widehat\tL_{k-2})$. The statement now follows by induction. 
\end{proof}

Let $\tt{i}_n\colon \tW_n\hookrightarrow \widehat{\tL_n} $ be the obvious inclusion. Taking multipath complexes is functorial with respect to  morphisms of digraph~\cite[Remark~6.3]{zbMATH07680365}, hence we get an induced map $\iota_n\colon X(\tW_n)\to  X(\widehat{\tL_n})$.

\begin{lem}\label{lem:hequivbln}
    For $n\equiv 0,3 \mod 4$, the multipath complexes $ X(\tW_n)$ and $  X(\widehat{\tL_n})$ are homotopy equivalent.
\end{lem}

\begin{proof}
    First, assume $n=3$. Then, it is easy to see that $X(\tW_3)$ and $X(\widehat{\tL_3})$ are homotopy equivalent to $S^1$. Furthermore, such homotopy equivalence is induced by the inclusion of $\tW_3$ in $\widehat{\tL_3}$ (see Figure~\ref{fig:hatBL3 e W3}). 
        Assume $n = 4k+3$. Then, the graph $\tW_n$ is obtained from $\tW_3$ by iterating $2k$-times the $T$-operation. Analogously, $\widehat{\tL_n}$ is obtained from $\widehat{\tL_3}$ by iterating $2k$-times the $T$-operation.
As taking suspensions is commutative up to homotopy, the statement for $n\equiv 3\mod 4$ follows.    The case $n\equiv 0\mod 4$ is completely analogous.
\end{proof}

    \begin{figure}[h]
        \centering
        \begin{tikzpicture}[scale = 1.5]
            \draw[fill,blue] (0,0) circle (.025);
            \draw[fill, opacity =.3] (1,0) circle (.025);
            \draw[fill,blue] (0.5,0.866) circle (.025);

            \draw[fill, blue] (2,0) circle (.025);
            \draw[fill, blue] (3,0) circle (.025);
            \draw[fill, blue] (2.5,0.866) circle (.025);
            
            \draw[fill,blue] (2,0.75) circle (.025);
        
            \node[above left] (v21) at (0.5,0.866) {$(v_2,v_1)$};
            \node[below] (v23) at (2,0) {$(v_2,v_3)$};
            \node[right] (v12) at (3,0) {$(v_1,v_2)$};
            \node[below left] (v32) at (0,0) {$(v_3,v_2)$};
            \node[below] (v10) at (1,0) {$(v_1,v_0)$};
            \node[above] (v34) at (2,0.75) {$(v_3,v_4)$};
            \node[above right] (v01) at (2.5,0.866) {$(v_0,v_1)$};

            \draw[pattern = north east lines, pattern color = green,opacity = .3] (0,0) -- (0.5,0.866) -- (1,0) -- cycle;
            
            \draw[pattern = north east lines, pattern color = green,opacity = .3] (2,0.75) -- (0.5,0.866) -- (1,0) -- cycle;

            \draw[pattern = north east lines, pattern color = green,opacity = .3] (2,0.75) -- (2,0) -- (1,0) -- cycle;
            
            \draw[blue, dashed] (2,0.75) -- (3,0);
            
            \draw[pattern = north east lines, pattern color = blue,opacity = .6] (2,0) -- (2.5,0.866) -- (3,0) -- cycle;
            \draw[blue] (2,0) -- (2.5,0.866) -- (3,0) -- cycle;

            \draw[pattern = north east lines, pattern color = blue,opacity = .5] (2,0.75) -- (2.5,0.866) -- (2,0) -- cycle;
            \draw[blue] (2,0.75) -- (2.5,0.866) -- (2,0) -- cycle;

            \draw[very thick, blue] (0,0) .. controls +(-.5,0) and +(-1,-1) .. (0,1.73) .. controls +(1,1) and +(0,1) .. (2.5,0.866) ;

            \draw[blue] (2,0.75) --   (2,0);
            \draw[blue] (2,0.75) --   (.5,.866);
            \draw[blue] (0,0) --   (.5,.866);

            \node at (1.5, -1) {(b)};
            
            \begin{scope}[shift ={+(-5,0)}]
                \draw[fill] (0,0) circle (.025);
            \draw[fill] (1,0) circle (.025);
            \draw[fill] (0.5,0.866) circle (.025);

            \draw[fill] (2,0) circle (.025);
            \draw[fill] (3,0) circle (.025);
            \draw[fill] (2.5,0.866) circle (.025);
            
            \draw[fill] (2,0.75) circle (.025);
        
            \node[above left] (v21) at (0.5,0.866) {$(v_2,v_1)$};
            \node[below] (v23) at (2,0) {$(v_2,v_3)$};
            \node[right] (v12) at (3,0) {$(v_1,v_2)$};
            \node[below left] (v32) at (0,0) {$(v_3,v_2)$};
            \node[below] (v10) at (1,0) {$(v_1,v_0)$};
            \node[above] (v34) at (2,0.75) {$(v_3,v_4)$};
            \node[above right] (v01) at (2.5,0.866) {$(v_0,v_1)$};
            \draw[dashed] (2,0.75) -- (3,0);
            
            \draw[pattern = north east lines, pattern color = green,opacity = .6] (2,0) -- (2.5,0.866) -- (3,0) -- cycle;
            \draw (2,0) -- (2.5,0.866) -- (3,0) -- cycle;
            
            \draw[pattern = north east lines, pattern color = green,opacity = .5] (0,0) -- (0.5,0.866) -- (1,0) -- cycle;
            \draw(0,0) -- (0.5,0.866) -- (1,0) -- cycle;
            
            \draw[pattern = north east lines, pattern color = green,opacity = .5] (2,0.75) -- (2.5,0.866) -- (2,0) -- cycle;
            \draw (2,0.75) -- (2.5,0.866) -- (2,0) -- cycle;
            
            \draw[pattern = north east lines, pattern color = green,opacity = .5] (2,0.75) -- (0.5,0.866) -- (1,0) -- cycle;
            \draw  (2,0.75) -- (0.5,0.866) -- (1,0) -- cycle;

            \draw[pattern = north east lines, pattern color = green,opacity = .5] (2,0.75) -- (2,0) -- (1,0) -- cycle;
            \draw  (2,0.75) -- (2,0) -- (1,0) -- cycle;

            \draw[very thick] (0,0) .. controls +(-.5,0) and +(-1,-1) .. (0,1.73) .. controls +(1,1) and +(0,1) .. (2.5,0.866) ;
            \node at (1.5, -1) {(a)};
            \end{scope}
        \end{tikzpicture}
        \caption{(a) The simplicial complex $X(\widehat{\tL}_3)$ and (b) the subcomplex $\iota_3(X(\tW_3))$ highlighted in blue. The vertices $(v_0,v_1)$,  $(v_1,v_2)$, $(v_2,v_3)$, and  $(v_3,v_4)$ span a simplex in both complexes.}
        \label{fig:hatBL3 e W3}
    \end{figure}
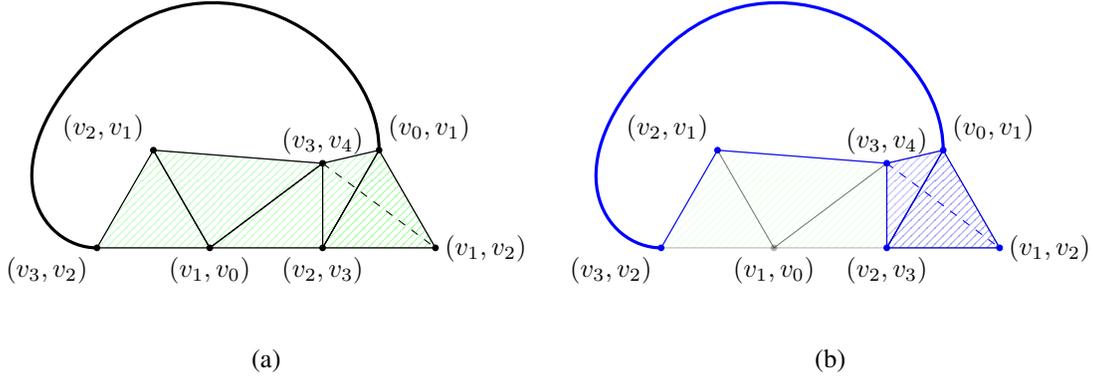

Observe that, arguing as in Lemma~\ref{lem:hequivbln},  we can not get a homotopy equivalence also for $n\equiv 1,2\mod 4$. In fact, as $X(\tW_2)\simeq S^0$ and $X(\widehat{\tL_2})\simeq \ast$, we get that $X(\tW_n)$ and $X(\widehat{\tL_n})$ are never homotopy equivalent for $n\equiv 2\mod 4$. Similarly for $n\equiv 1\mod 4$. However, one can still iteratively apply T operations, getting the following:

\begin{prop}\label{prop:wn}
    For $n\geq 1$ natural number, we have:
\[
X(\tW_n) \simeq \begin{cases} *  & \text{if }n \equiv 0,1 \mod 4\\
{S}^{\lfloor \frac{n -1}{2}\rfloor}  & \text{if }n \equiv 2,3 \mod 4\\
\end{cases}.\]
    In particular, the multipath complex $X(\tW_n)$ is either a sphere or it is contractible.
\end{prop}

\begin{proof}
The cases $n\equiv 0,3\mod 4$ directly follow from  Corollary~\ref{cor:BLnhat} and Lemma~\ref{lem:hequivbln}. The case $n\equiv 1\mod 4$ is obtained by iterating T-operations on the oriented linear simple path with two edges; which is contractible. To conclude, observe that $|X(\tW_2)|=S^0$ and that, for $n\equiv 2\mod 4$, $\tW_n$ is obtained by iterating the T-operation starting from~$\tW_2$. If~$n=4k+2$ the number of T-operation we need to perform is $k+1$. Which means, that $X(\tW_n)$ is homotopy equivalent to the $2k$-fold suspension over $S^0$, and thus $X(\tW_n)\simeq S^{2k}$ as desired.
\end{proof}

\subsection{Bidirectional polygonal digraphs: connectivity and homology}

Next we obtain some bounds on the degree of connectedness for the multipath complex of bidirectional polygonal digraphs.
An immediate consequence of Proposition~\ref{prop:connectivitylin} is that $X(\tL_n)$ is $\nu_n$-connected, where $\nu_n=\lfloor \frac{n-1}{2}\rfloor -1$. We use this fact to prove that the multipath complex of bidirectional polygonal digraphs are also highly connected. For all $n\geq 3$, let $\tP_n$ denote the bidirectional polygonal digraph with $n{+1}$ vertices ${0,}1,\dots, n$ and $2n {+ 2}$ edges $(i,i+1)$ and $(i+1,i)$ for all $i$ taken modulo~$n {+1}$ -- see Figure~\ref{fig:polydash}. This can be obtained from $\tL_{n+1}$ by identifying the external vertices.

To provide a connectivity bound, we use a more general version of Lemma~\ref{lem:BjornerLemma}:

\begin{lem}[Nerve Lemma]\label{lemma:nervelem}
Let $\Delta$ be a simplicial complex and $\{L_i\}_{i=1}^n$ a family of subcomplexes covering $\Delta$. Suppose that every non-empty intersection $L_{i_1}\cap\dots\cap L_{i_t}$ is $(k-t+1)$-connected, for $t\geq 1$. Then, $\Delta$ is $k$-connected if and only if the nerve of the covering  $\{L_i\}_{i=1}^n$ is $k$-connected.
\end{lem}

For a proof of this version of the nerve lemma, we refer to \cite[Lemma~1.2]{MR1253009}.

\begin{thm}\label{thm:highconn}
The multipath complex of the  digraph $\tP_{n}$  is $\nu_{n}$-connected.
\end{thm}

\begin{proof}
For the bidirectional polygon $\tP_n$ on $n$ vertices, denote by $e_i, e_i'$ the edges connecting the vertices $i$ and $i+1$; we use the convention $e_i\coloneqq (i,i+1)$ and $e_i'\coloneqq (i+1,i)$. We now decompose the multipath complex $\tP_n$, for $n\geq {2}$, as a union of multipath complexes of the digraphs obtained by deleting the edges $e_i,e_i'$ from $\tP_n$. To be more precise, consider the multipath complex $X_i\coloneqq X(\tP_n\setminus \{e_i,e_i'\})$; it is clear that the family $\{X_i\}_{i=1}^n$ yields a covering of $X(\tP_n)$.  Furthermore, note that the nerve of this covering is $(n-2)$-connected. In fact,  it has $n{+1}$ vertices and all intersections, but the maximal one, are non-empty; topologically, the nerve is obtained from an $n$-simplex by removing the interior, in turn yielding an $(n-1)$-dimensional sphere.

Note that the complex $X_i$ can be identified with $ X(\tL_{n})$, the multipath complex of the bidirectional linear graph on~$n$ vertices. Hence, by Proposition~\ref{prop:connectivitylin}, each multipath complex $X_i$  is $\nu_{n}$-connected. 
If we show that \[ Y_h = X_{i_1}\cap \dots \cap X_{i_h}\] is $(\nu_{n}-h+1)$-connected for all $h\geq 1$ and $i_1 < i_2 < ...< i_h$, then, by Lemma~\ref{lemma:nervelem}, also $X(\tP_n)$ is $\nu_{n}$-connected -- as the nerve of the associated covering is $(n-2)$-connected, and $\nu_{n}\leq n-2$. To this end, observe that $ Y$ is a multipath complex of a disjoint union of bidirectional linear graphs, say $\tL_{r_1},...,\tL_{r_s}$.
We know that the multipath complex of a disjoint union of graphs is the join of the corresponding multipath complexes. It follows that
\[ Y_h\simeq S^{\nu_{r_1}{+1}}  \ast \cdots \ast  S^{\nu_{r_s}{+1}} \simeq   S^{\sum_{j} \nu_{r_j} + {2}s -1} \ .\]
We can give a bound for the dimension of the above sphere, in fact:
\[ c : = \sum_{j=1}^s \nu_{r_j} + s =\sum_{j=1}^{s} \left\lfloor \frac{r_j -1}{2} \right\rfloor = \sum_{r_j\text{ odd }} \frac{r_j - 1}{2} + \sum_{r_j\text{ even }} \frac{r_j - 2}{2} =   \]
\[ = \sum_{j=1}^{s} \frac{r_j -1}{2} - \frac{1}{2} \# \{ r_j\text{ even }\} \geq  \]
\[\geq \sum_{j=1}^{s} \frac{r_j -1}{2} - \frac{s}{2} \overset{(\star)}{=}\]
\[ \overset{(\star)}{=} \frac{n {+1} -h -s}{2} - \frac{s}{2} = \frac{n - 1 + 2}{2} - h + \frac{h}{2} - s \geq \left\lfloor \frac{n-1}{2} \right\rfloor - h -s +1 +\frac{h}{2}.\]
In the equality marked with ($\star$) we used that $\sum_j r_j = n+1-h$.
Since $c$ is an integer and $h>0$, it follows that~$c \geq \nu_{n} + 1 - h -s + 2$.
Therefore, the complex $Y_h$ is at least $c + s - 2 = (\nu_{n}-h + 1)$-connected for each $h\geq 1$, concluding the proof. 
\end{proof}

\begin{rem}
Note that the worst connectivity bound of the intersections $X_{i_1}\cap \dots \cap X_{i_h}$ in the proof of the above proposition is achieved when, after reordering of the indices, the intersection yields a \emph{connected} bidirectional linear digraph. 
\end{rem}

\begin{figure}[h]
\centering
\newdimen\R
\R=2.00cm

\newdimen\Rb
\Rb=1.75cm
\begin{tikzpicture}
\draw[xshift=5.0\R, fill] (270:\R) circle(.05)  node[below] {$v_0$};
\draw[xshift=5.0\R,fill] (225:\R) circle(.05)  node[below left]   {$v_1$};
\draw[xshift=5.0\R,fill] (180:\R) circle(.05)  node[left] {$v_2$};
\draw[xshift=5.0\R,fill] (135:\R) circle(.05)  node[above left] {$v_3$};
\draw[xshift=5.0\R, fill] (90:\R) circle(.05)  node[above] {$v_4$};
\draw[xshift=5.0\R,fill] (45:\R) circle(.05)  node[above right] {$v_5$};
\draw[xshift=5.0\R,fill] (0:\R) circle(.05)  node[right] {$v_6$};
\draw[xshift=5.0\R,fill] (315:\R) circle(.05)  node[below right] {$v_{n}$};

\node[xshift=5.0\R] (v0) at (270:\R) { };
\node[xshift=5.0\R] (v1) at (225:\R) { };
\node[xshift=5.0\R] (v2) at (180:\R) { };
\node[xshift=5.0\R] (v3) at (135:\R) { };
\node[xshift=5.0\R] (v4) at (90:\R) { };
\node[xshift=5.0\R] (v5) at (45:\R) { };
\node[xshift=5.0\R] (v6) at (0:\R) { };
\node[xshift=5.0\R] (vn) at (315:\R) { };

\draw[thick, blue, -latex] (v0) to[bend left] (v1);
\draw[thick, blue, -latex] (v1) to[bend left] (v2);
\draw[thick, blue, -latex] (v2) to[bend left] (v3);
\draw[thick, blue, -latex] (v3) to[bend left] (v4);
\draw[thick, blue, -latex] (v4) to[bend left] (v5);
\draw[thick, blue, -latex] (v5) to[bend left] (v6);
\draw[thick, blue, -latex] (vn) to[bend left] (v0);

\draw[thick, red, latex-] (v0) to[bend right] (v1);
\draw[thick, red, latex-] (v1) to[bend right] (v2);
\draw[thick, red, latex-] (v2) to[bend right] (v3);
\draw[thick, red, latex-] (v3) to[bend right] (v4);
\draw[thick, red, latex-] (v4) to[bend right] (v5);
\draw[thick, red, latex-] (v5) to[bend right] (v6);
\draw[thick, red, latex-] (vn) to[bend right] (v0);

\draw[xshift=5.0\R, fill] (292.5:\Rb) node[above left] {$e'_{n}$};
\draw[xshift=5.0\R,fill] (247.5:\Rb) node[above right] {$e'_0$};
\draw[xshift=5.0\R,fill] (202.5:\Rb)   node[above right] {$e'_1$};
\draw[xshift=5.0\R,fill] (157.5:\Rb)  node[below right] {$e'_2$};
\draw[xshift=5.0\R, fill] (112.5:\Rb)   node[below right] {$e'_3$};
\draw[xshift=5.0\R,fill] (67.5:\Rb) node[below left] {$e'_4$};
\draw[xshift=5.0\R,fill] (22.5:\Rb) node[below left] {$e'_5$};

\draw[xshift=5.0\R, fill] (292.5:\R) node[below right] {$e_{n}$};
\draw[xshift=5.0\R,fill] (247.5:\R) node[below left] {$e_0$};
\draw[xshift=5.0\R,fill] (202.5:\R)   node[below left] {$e_1$};
\draw[xshift=5.0\R,fill] (157.5:\R)  node[above left] {$e_2$};
\draw[xshift=5.0\R, fill] (112.5:\R)   node[above left] {$e_3$};
\draw[xshift=5.0\R,fill] (67.5:\R) node[above right] {$e_4$};
\draw[xshift=5.0\R,fill] (22.5:\R) node[above right] {$e_5$};
\draw[xshift=4.95\R,fill] (337.5:\R)  node {$\cdot$} ;
\draw[xshift=4.95\R,fill] (333:\R)  node {$\cdot$} ;
\draw[xshift=4.95\R,fill] (342:\R)  node {$\cdot$} ;
\end{tikzpicture}
\caption{A bidirectional polygonal digraph on $2n$ edges with (at least) two vertices that are neither sources nor sinks (in blue). The dashed line shows the separation between the two modules.} 
\label{fig:polydash}
\end{figure}
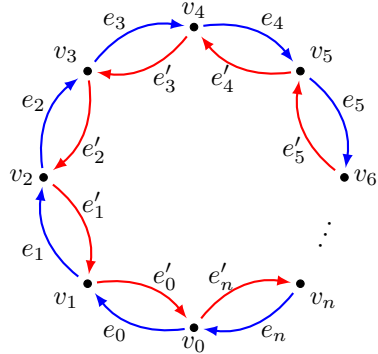

Using the Mayer-Vietoris spectral sequence, we can compute the whole homology of bidirectional polygonal digraphs. 

\begin{thm}\label{thm:connectivity_homology}
Let $X(\tP_n)$ be the multipath complex of the digraph $\tP_n$, for $n\geq 2$. Then, we have 
\[
\H_i(X(\tP_n)) = \begin{cases} 
\bZ  & \text{ if } i=0 \\
\bZ^2  & \text{ if } i=n-1 \\
\bZ  & \text{ if } i={\lfloor \frac{n - 1}{2} \rfloor} \text{ and } n\equiv 1,2 \mod 4\\
\bZ^3 & \text{if }   i={\lfloor \frac{n - 1}{2} \rfloor} \text{ and } n \equiv 3 \mod 4\\
\bZ  & \text{ if } i={\lfloor \frac{n - 1}{2} \rfloor}+1 \text{ and } n\equiv 0 \mod 4
\end{cases}.\]
In particular, $\tP_n$ is not $(\nu_n+1)$-connected for $n\equiv 1,2,3 \mod 4$.  
\end{thm}

\begin{proof}
We start by considering three sub-complexes of the simplicial complex $X = X(\mathtt{BP}_{n})$, namely: $A = {\rm star}(e_0)$, $B = {\rm star}(e'_0)$, and $C = {\rm astar}(e_0) \cap {\rm astar}(e'_0)$. 
Note that the complement of $A$ is a subset of $B \cup C$, and therefore $X = A \cup B \cup C$. By construction, $A$ and $B$ are contractible. 
While $C$ is the multipath complex of $\tL_{n}$ -- spanned by the edges~$e_1,...,e_{n},e_1',...,e_{n}'$.
It remains to analyse the intersections. Clearly, $A\cap C \cong B\cap C$ (as simplicial complexes) and they both are isomorphic to the multipath complex of $\tW_{n-1}$ with the interior of a maximal cell $\Delta$ spanned by~$e_1,\dots , e_n$ removed (because it is not in $A$). Analogously for $B$. 
Finally, $A \cap B =  A \cap B \cap C$ can be identified with $\tL_{n-2}$.

By Proposition~\ref{prop:connectivitylin}, $C$ is homotopy equivalent to $S^{\lfloor \frac{n - 1}{2} \rfloor}$ and both $A\cap B$ and $A\cap B \cap C$ are homotopy equivalent to $S^{\lfloor \frac{n - 3}{2} \rfloor}$. 
Now, using Proposition~\ref{prop:wn}, and the Mayer-Vietoris long exact sequence in homology applied to the triple~$(\tW_{n-1},\,\tW_{n-1}\setminus\Delta,\, \Delta)$, we obtain 
\[
H_*(A\cap C) \cong H_*(B\cap C ) 
\cong \begin{cases} \bZ   & \text{if }* = 0,\, n -2\\
\bZ
  & \text{if }* = \lfloor \frac{n -2}{2}\rfloor\text{ and } n \equiv 0,3 \mod 4 \\
\end{cases}.\]
We can put it into the first page of the Mayer-Vietoris spectral sequence associated to the cover $\{A,B,C\}$. When $n\equiv 1,2 \mod 4$, we get:
\begin{equation*}
\begin{tikzpicture}
  \matrix (m) [matrix of math nodes,
    nodes in empty cells,nodes={minimum width=5ex,
    minimum height=5ex,outer sep=-5pt},
    column sep=1ex,row sep=1ex]{
          q      &      &     &     & E^1\\
          {n-2} \qquad    &  0 &  \bZ^2   & 0 & \\
          {\lfloor \frac{n - 1}{2} \rfloor} \qquad    &  \bZ &  0  & 0 & \\
          {\lfloor \frac{n - 3}{2} \rfloor}  \qquad   &  0  & \bZ  &  \bZ  & \\
    \quad\strut &   0  &  1  &  2  & \strut p\\};
\draw[thick] (m-1-1.east) -- (m-5-1.east) ;
\draw[thick] (m-5-1.north) -- (m-5-5.north) ;
\end{tikzpicture}
\quad
\begin{tikzpicture}
  \matrix (m) [matrix of math nodes,
    nodes in empty cells,nodes={minimum width=5ex,
    minimum height=5ex,outer sep=-5pt},
    column sep=1ex,row sep=1ex]{
          q      &      &     &     & E^2\\          
          {n-2} \qquad    &  0 &  \bZ^2   & 0 & \\
          {\lfloor \frac{n - 1}{2} \rfloor} \qquad     & \node(b) {\bZ}; &  0  & 0 & \\
          {\lfloor \frac{n - 3}{2} \rfloor} \qquad      &  0  & 0  &0 & \\
    \quad\strut &   0  &  1  &  2  & \strut p\\};
\draw[thick] (m-1-1.east) -- (m-5-1.east) ;
\draw[thick] (m-5-1.north) -- (m-5-5.north) ;
\end{tikzpicture}
\end{equation*}
The inclusion $A\cap B\cap C$ into $A\cap B$ is always a homeomorphism, which kills the homology classes corresponding to those at the second page. When $n\equiv 0 \mod 4$ we get $\lfloor \frac{n - 1}{2} \rfloor=\lfloor \frac{n - 2}{2} \rfloor$, and the pages of the spectral sequence look like:
\begin{equation*}
\begin{tikzpicture}
  \matrix (m) [matrix of math nodes,
    nodes in empty cells,nodes={minimum width=5ex,
    minimum height=5ex,outer sep=-5pt},
    column sep=1ex,row sep=1ex]{
          q      &      &     &     & E^1\\
          {n-2} \qquad    &  0 &  \bZ^2   & 0 & \\
          {\lfloor \frac{n - 1}{2} \rfloor} \qquad    &  \bZ &  \bZ^2  & 0 & \\
          {\lfloor \frac{n - 3}{2} \rfloor}  \qquad   &  0  & \bZ  &  \bZ  & \\
    \quad\strut &   0  &  1  &  2  & \strut p\\};
\draw[thick] (m-1-1.east) -- (m-5-1.east) ;
\draw[thick] (m-5-1.north) -- (m-5-5.north) ;
\end{tikzpicture}
\quad
\begin{tikzpicture}
  \matrix (m) [matrix of math nodes,
    nodes in empty cells,nodes={minimum width=5ex,
    minimum height=5ex,outer sep=-5pt},
    column sep=1ex,row sep=1ex]{
          q      &      &     &     & E^2\\          
          {n-2} \qquad    &  0 &  \bZ^2   & 0 & \\
          {\lfloor \frac{n - 1}{2} \rfloor} \qquad     & \node(b) {0}; &  \bZ  & 0 & \\
          {\lfloor \frac{n - 3}{2} \rfloor} \qquad      &  0  & 0  &0 & \\
    \quad\strut &   0  &  1  &  2  & \strut p\\};
\draw[thick] (m-1-1.east) -- (m-5-1.east) ;
\draw[thick] (m-5-1.north) -- (m-5-5.north) ;
\end{tikzpicture}
\end{equation*}
Using Eq.~\ref{eq:almostbid2}, observe that the first differentials for $q=\lfloor \frac{n - 1}{2} \rfloor$ of $A\cap C, B\cap C$ in $C$ are isomorphisms, hence in the second page there is only one homology group left in bidegree $(1, \lfloor \frac{n - 1}{2} \rfloor)$.

When $n\equiv 3 \mod 4$ , we have the equality $\lfloor \frac{n - 2}{2} \rfloor=\lfloor \frac{n - 3}{2} \rfloor$, and thus the first two pages of the spectral sequence are as follows:

\begin{equation*}
\begin{tikzpicture}
  \matrix (m) [matrix of math nodes,
    nodes in empty cells,nodes={minimum width=5ex,
    minimum height=5ex,outer sep=-5pt},
    column sep=1ex,row sep=1ex]{
          q      &      &     &     & E^1\\
          {n-2} \qquad    &  0 &  \bZ^2   & 0 & \\
          {\lfloor \frac{n - 1}{2} \rfloor} \qquad    &  \bZ &  0  & 0 & \\
          {\lfloor \frac{n - 2}{2} \rfloor}  \qquad   &  0  & \bZ^2\oplus \bZ  &  \bZ  & \\
    \quad\strut &   0  &  1  &  2  & \strut p\\};
\draw[thick] (m-1-1.east) -- (m-5-1.east) ;
\draw[thick] (m-5-1.north) -- (m-5-5.north) ;
\end{tikzpicture}
\quad
\begin{tikzpicture}
  \matrix (m) [matrix of math nodes,
    nodes in empty cells,nodes={minimum width=5ex,
    minimum height=5ex,outer sep=-5pt},
    column sep=1ex,row sep=1ex]{
          q      &      &     &     & E^2\\          
          {n-2} \qquad    &  0 &  \bZ^2   & 0 & \\
          {\lfloor \frac{n - 1}{2} \rfloor} \qquad     & \node(b) {\bZ}; &  0  & 0 & \\
          {\lfloor \frac{n - 2}{2} \rfloor} \qquad      &  0  & \bZ^2  &0 & \\
    \quad\strut &   0  &  1  &  2  & \strut p\\};
\draw[thick] (m-1-1.east) -- (m-5-1.east) ;
\draw[thick] (m-5-1.north) -- (m-5-5.north) ;
\end{tikzpicture}
\end{equation*}
In both cases the cancellation is due to the fact that $A\cap B\cap C = A\cap B$. This concludes the proof.
\end{proof}

Next we consider the effect on the homology of $\tP_n$ when edges are removed, to try and understand where the homological complexity appears compared to the cyclic graph $C_n$. 
 To this aim we keep all the edges of one orientation, such as all \emph{forward edges} $e_i$, and remove some subset of the \emph{backwards edges} $e'_i$.
 If we remove a single edge, or a pair of consecutive edges, we get the following conjectures, which have been checked computationally for all $n\le 14$:

\begin{conj}\label{conj:remove_edges1}
Let $\taP{-1}$ be the graph obtained by removing a single edge from $\tP_n$, then:
\[
\H_i(\taP{-1}) = \begin{cases} 
\bZ  & \text{ if } i=0 \\
\bZ  & \text{ if } i=n-1 \\
\bZ  & \text{ if } i={\lfloor \frac{n - 1}{2} \rfloor} \text{ and } n\equiv 1 \mod 4\\
\bZ & \text{if }   i={\lfloor \frac{n - 2}{2} \rfloor} \text{ and } n \equiv 2 \mod 4\\
\bZ^2  & \text{ if } i={\lfloor \frac{n - 1}{2} \rfloor}+1 \text{ and } n\equiv 3 \mod 4
\end{cases}.\]\end{conj}
\begin{conj}\label{conj:remove_edges2}
Let $\taP{-2}$ be obtained by removing two consecutive edges, of the same orientation, from $\tP_n$, then:
\[
\H_i(\taP{-2}) = \begin{cases} 
\bZ  & \text{ if } i=0 \\
\bZ  & \text{ if } i=n-1 \\
\bZ & \text{if }   i={\lfloor \frac{n - 2}{2} \rfloor} \text{ and } n \equiv 2 \mod 4\\
\bZ  & \text{ if } i={\lfloor \frac{n - 1}{2} \rfloor}+1 \text{ and } n\equiv 3 \mod 4
\end{cases}.\]
\end{conj}

If we remove at least three consecutive edges from $\tP_n$, then the multipath complex is the same as the multipath complex of oriented cycle graph $C_n$, i.e. an $n-1$ sphere.

\begin{prop}
    Let $\taP{-*}$ be the subgraph of $\tP_n$ with all edges of one orientation (i.e. $C_n$ is a subgraph), and has at least $3$ consecutive edges removed of the reverse orientation. Then $X(\taP{-*})$ is homotopy equivalent to an $n-1$ sphere.
    \begin{proof}
        Without loss of generality suppose the edges $(v_3,v_2),(v_2,v_1)$ and $(v_1,v_0)$ are not in $\taP{-*}$. Define the poset fibration between the multipath posets $f:P(\taP{-*}
        )\rightarrow P(C_n)$ which maps every multipath that contains a backwards edge to the multipath $p=(v_1,v_2)$, and acts as the identity on all other multipaths. 

        The fibres $f^{-1}(P(C_n)_{\le q})$ are contractible for all $q\not=p$, as each has the maximal element $q$. The fibre~$f^{-1}(P(C_n)_{\le q})$ is also contractible as it has the maximal element 
         $$(v_1,v_2),(v_4,v_3),(v_5,v_4),\ldots,(v_n,v_0).$$
        Thus, by Quillen's fibre theorem \cite[Proposition 1.6]{quillen1978homotopy}, $P(\taP{-*})$ is homotopy equivalent to $P(C_n)$, which is a sphere \cite[Example 6.9]{zbMATH07680365}.
    \end{proof}
\end{prop}

\section{Non-Shellability of bidirectional polygonal and linear graphs}

Whilst we have determined the homology of the multipath complexes of bidirectional polygonal digraphs, the question of their homotopy type remains open. We conjecture that the multipath complex of every bidirectional polygonal digraph is a wedge of spheres, and one natural approach to show this is using shellability, see \cite{wachs2006poset} for background on shellability. 
Unfortunately, this does not help in this case as we can show that $X(\tP_n)$ is not shellable.

\begin{prop}\label{prop:notshellBPn}
    The multipath complex $X(\tP_n)$ is not shellable, for all $n\ge2$.
    \begin{proof}
        Suppose there is a shelling order $F_1,F_2,\ldots,F_k$ of the maximal faces of $X(\tP_n)$, i.e. the maximal multipaths of $\tP_n$.
        Let $F_x$ be the first maximal face in the order with no backwards edges, as it is maximal such a path must consist of all forward edges except one. 
        The intersection of $F_x$ with $\cap_{i<x}F_i$ must have dimension at least 2 less than $F_x$, as any preceding maximal face must contain a backward edge, so be missing at least 3 forward edges.
        Yet a shelling order require this intersection has dimension 1 less. So to satisfy this condition the ordering must begin with a maximal path with no backwards edges.
        However, by symmetry the same argument implies that the shelling order must begin with a maximal path with no forward edges. 
        Clearly both cannot hold as the only such path is the empty path, which is not maximal in $\tP_n$ with $n\ge2$, so no shelling order exists.
    \end{proof}
\end{prop}

Thus we leave it as an open problem, as to whether the multipath complex of bidirectional polygonal digraphs is a wedge of spheres.
We have shown in Proposition~\ref{prop:connectivitylin} that the multipath complex of the bidirectional linear graph $X(\tL_n)$ is homotopy equivalent to a sphere. However, by a similar argument to Proposition~\ref{prop:notshellBPn} we can show that $X(\tL_n)$ is not shellable:

\begin{prop}
    The multipath complex $X(\tL_n)$ is not shellable, for all $n\ge2$.
    \begin{proof}
        Let $p$ be the directed path containing all forward edges, i.e. edges $(v_i,v_{i+1})$. 
        This path must appear first in the ordering, as every other maximal multipath must contain at least two fewer forward edges. 
        However, by symmetry the same applies to the path containing all backwards edges. So no shelling order exists.
    \end{proof}
\end{prop}

\bibliographystyle{alpha}
\bibliography{biblio}
\end{document}